\documentclass{amsart}

\usepackage{ifthen}
\usepackage{graphicx}
\usepackage{amssymb}
\usepackage{enumerate}
\usepackage{url}
\usepackage{paralist}
\usepackage{charter}
\usepackage{stmaryrd}
\usepackage{mathabx}
\usepackage{stmaryrd}
\usepackage{textcomp}
\usepackage{cmll}
\usepackage{mathrsfs}
\usepackage{url}
\usepackage{MnSymbol}

\newtheorem{anyprop}{Anyprop}[section]

\newtheorem{theorem}[anyprop]{Theorem}

\theoremstyle{definition}

\theoremstyle{remark}

\numberwithin{equation}{section}

\usepackage{amscd}
\usepackage{amssymb}
\input xy
\xyoption{all}

\begin{document}
\title[DATHEMATICS]
{DATHEMATICS: A META-ISOMORPHIC VERSION OF `STANDARD' MATHEMATICS BASED ON PROPER CLASSES}

\author[Danny Arlen de Jes\'us G\'omez-Ram\'irez]{Danny Arlen de Jes\'us G\'omez-Ram\'irez}
\address{Vienna University of Technology, Institute of Discrete Mathematics and Geometry,
wiedner Hauptstrasse 8-10, 1040, Vienna, Austria.}
\email{daj.gomezramirez@gmail.com}

\begin{abstract}
We show that the (typical) quantitative considerations about proper (as too big) and small classes are just tangential facts regarding the consistency of Zermelo-Fraenkel Set Theory with Choice. Effectively, we will construct a first-order logic theory D-ZFC (Dual theory of ZFC) strictly based on (a particular sub-collection of) proper classes with a corresponding special membership relation, such that ZFC and D-ZFC are meta-isomorphic frameworks (together with a more general dualization theorem). More specifically, for any standard formal definition, axiom and theorem that can be described and deduced in ZFC, there exists a corresponding `dual' version in D-ZFC and vice versa. Finally, we prove the meta-fact that (classic) mathematics (i.e. theories grounded on ZFC) and dathematics (i.e. dual theories grounded on D-ZFC) are meta-isomorphic. This shows that proper classes are as suitable (primitive notions) as sets for building a foundational framework for mathematics.


\end{abstract}
\maketitle

\smallskip


\section*{Introduction}

At the beginning of the twentieth century there was a particular interest among mathematicians and logicians in finding a general, coherent and consistent formal framework for mathematics. One of the main reasons for this was the discovery of paradoxes in Cantor's Naive Set Theory and related systems, e.g., Russell's, Cantor's, Burati-Forti's, Richard's, Berry's and Grelling's paradoxes \cite{mendelsonlogic}, \cite{curry}, \cite{tait}, \cite{church}, \cite{french} and \cite{martin}. 
In particular, Russell's paradox offered one of the strongest motivations for developing new and more restricted set-theoretical frameworks. Specifically, the seminal works of E. Zermelo \cite{zermelo}; A. Fraenkel \cite{fraenkel}; J. von Newmann  \cite{vonnewmann}; P. Bernays \cite{bernays}, \cite{bernays2}; R. Robinson \cite{robinson}; and K. Goedel \cite{goedel1}, \cite{goedel2}, \cite{goedel3}; allow for the construction of the most accepted and well-known logical formal frameworks of Zermelo-Fraenkel Set Theory with Choice (ZFC) \cite{jech}, and more generally Von Newmann-Bernays-Goedel Set Theory (NBG) \cite[Ch. 4]{mendelsonlogic}. 

Now, in the context of NBG set theory the essential starting point was the intuitive idea that a kind of new entity called a `proper class' should be formed from the general  collection of all sets, because this special collection was `too big'. So, the general framework of NBG set theory is based on the primitive notion of class and the primitive relation of membership among classes. In addition, the notion of set is captured by restricting the classes to those who belong to at least another class. So, in this way one can guarantee with a suitable axiomatization that such classes remain small enough in order to prevent contradictory statements like Russell's paradox, and to fulfill the main axioms of ZFC set theory required for constructing the most fundamental mathematical theories e.g. analysis, (differential and algebraic) geometry, (abstract) algebra and number theory. 

In addition, an implicit working principle in NBG set theory is that small classes (or `sets') are more suitable objects to start and work with. On the other hand, proper classes are just too big and formally `too dangerous' in order to be able to ground any consistent and enough general mathematical theory. 

In this paper, we will mainly show that these classic quantitative considerations about proper and small classes are just tangential facts regarding the consistency of ZFC set theory. Effectively, we will construct a logic theory D-ZFC (Dual theory of ZFC set theory) strictly based on (a particular sub-collection of) proper classes with a corresponding special membership relation, such that ZFC and D-ZFC are meta-isomorphic frameworks. More specifically, for any standard formal definition, axiom and theorem that can be described and deduced in ZFC set theory, there exists a corresponding `dual' version in D-ZFC and vice versa. In particular ZFC set theory is consistent if and only if D-ZFC is consistent.

\section{Dual Notions and Axioms of Zermelo-Fraenkel Set Theory with Choice within NGB Set Theory}
\label{sect1}
In this section we will follow the treatment of E. Mendelson on the construction of the whole framework for classes and sets developed in NBF set theory \cite[Ch. 4]{mendelsonlogic}. 

Von Newmann-Bernays-G\"odel Set Theory is a very special framework in the sense that it allows the existence of complementary classes, which can be seen as `dual' classes regarding the meta-class of all classes. Specifically, we will use the formal symmetry lying in the Axiom of the Existence of the Complement Class, which asserts that for any class $X$, there exists a (unique) \emph{dual} class $X^+$ satisfying

\[ (\forall a)(a\in X \leftrightarrow a\notin X^+),\]

where $a$ varies over sets \cite[Ch. 4, B4]{mendelsonlogic}.


We will define dual notions of the main structural concepts of NBG Set Theory based on the former axiom.

Let us start with the dual notion of the membership relation $\in$, which we denote by $\varepsilon$. 
This \emph{dembership relation} is defined by the following axiom:\footnote{In most of the cases the name of the dual notions will be given by replacing (resp. adding to) the first letter of the original name with the letter `d', coming from `dual'. For example, the dual of the membership relation is called `dembership relation'.}

\[ (\forall A,B)(A\varepsilon B\leftrightarrow A^+\in B^+).\]

In this case, we say that $A$ is a \emph{dember} (\emph{delement}) of $B$.

For the dual notion of set, we analyze the corresponding dual formula:
\[M_d(X):\Leftrightarrow (\exists Y)(X \varepsilon Y)\Leftrightarrow (\exists Y)(X^+ \in Y^+)\]

\[\Leftrightarrow (\exists Z)(X^+\in Z).\]
So what it means is that $X^+$ is a \emph{sed} (dual set), if and only if, its complement $X^+$ is a set, since $Y$ varies over all classes, if and only if, $Y^+$ varies over all classes.

Now, let us prove that there is a `dual' theory of NBG set theory based on a special sub-collection of proper classes playing the dual role that sets play in NBG:


\subsection{Dual Notion of Equality}

The notion of equality for classes and its dual are exactly the same:

\[X=_dY :\Leftrightarrow (\forall Z)(Z\varepsilon X \leftrightarrow Z\varepsilon Y)\]
\[ \Leftrightarrow (\forall Z^+)(Z^+\in X^+ \leftrightarrow Z^+\in Y^+)\]
\[ \Leftrightarrow (\forall W)(W\in X^+ \leftrightarrow W\in Y^+):\Leftrightarrow X^+=Y^+\Leftrightarrow X=Y\]

\subsection{Dual Inclusion}

 The dual notion of inclusion, namely, \emph{dinclusion} is defined as usual:

\[X\sqsubseteq Y :\Leftrightarrow (\forall Z)(Z\varepsilon X \rightarrow Z\varepsilon Y).\]

We express this by saying that $X$ is a \emph{subsed} of $Y$.




\subsection{Dual Proper Classes}

 The dual notion of proper class is called \emph{d-proper class} and is given by 
\[\neg M_d(W) \Leftrightarrow (\forall Y)(\neg(W\varepsilon Y))\Leftrightarrow\]
\[(\forall Y) \neg(W^+ \in Y^+)\Leftrightarrow (\forall Z)(\neg(W^+\in Z)),\]

where $Z=Y^+$ varies over all classes. So, $W$ is a d-proper class if and only if $W^+$ is a proper class.

Informally, seds have very similar properties as sets, when replacing $\in$ by $\varepsilon$. 

In addition, since one of the central notions of NBG set theory is the concept of set, we want to understand its behavior within the framework of the $\varepsilon$ relation. So, we will focus our attention on the dual versions of the further axioms regarding sets.

\subsection{Dual Axiom T}

 The dual version of the Axiom T, namely, the \emph{Axiom $T^+$} coincides with the corresponding Axiom T due to the following reasons:

\[X=Y \Leftrightarrow X^+=Y^+ \Rightarrow (\forall Z)(X^+\in Z \leftrightarrow Y^+\in Z)\]

\[\Leftrightarrow (\forall W)(X^+\in W^+\leftrightarrow Y^+\in W^+)\]

\[\Leftrightarrow (\forall W)(X\varepsilon W \leftrightarrow Y\varepsilon W).\]

The last chain of equivalences hold due to the fact that $W$ and $Z$ vary over all classes, if and only if, $W^+$ and $Z^+$ so too do.

Besides, it is clear that

\[(\forall A,B)(A=B \leftrightarrow A^+=B^+).\]

In conclusion, the Axiom $T^+$ states 

\[X=Y \Rightarrow (\forall W)(X\varepsilon W \leftrightarrow Y\varepsilon W).\]


\subsection{Dual Predicative Well-formed Formulas}

 We denote sed variables (i.e., symbols which vary only over seds) by lower-case letters and classes by upper-case letters. So a \emph{dual predicative well-formed} (dwf) formula is just a w. f. formula $\Phi$, where all the bound variables are sed variables.

\subsection{Dual Pairing Axiom}

The dual version of the Pairing Axiom, namely, \emph{Axiom $P^+$} is the following:

\[(\forall x)(\forall y)(\exists z)(\forall u)(u\varepsilon z \leftrightarrow u = x \vee u = y).\]







Now, it is equivalent to the sentence:

\[(\forall x)(\forall y)(\exists z)(\forall u)(u^+\in z^+ \leftrightarrow u^+ = x^+ \vee u^+ = y^+).\]

Besides, $x,y,z$ and $u$ vary over seds if and only if $x^+,y^+,z^+$ and $u^+$ vary over sets. So, the last expression is equivalent to

\[(\forall x^+)(\forall y^+)(\exists z^+)(\forall u^+)(u^+\in z^+ \leftrightarrow u^+ = x^+ \vee u^+ = y^+),\]

where all the variables appearing here are set variables. So, if we know that all symbols $\Xi^+$ vary over sets, then we could eliminate the symbols $(-)^+$ and obtain, in fact, just the classic pairing axiom of NBG. So, the Axiom $P^+$ is just stating that $z^+=\{ x^+,y^+\}$, and we will denote this by $z=\Lbag x,y\Rbag$. In other words, the Axiom $P^+$ states that for any seds $x$ and $y$, there exists a (uniquely determined) sed $z$ having as \emph{denements} exactly $x$ and $y$.\footnote{In this section we show explicitly the essential constructions and (in some sense similar) arguments due basically to achieve an axiomatic completeness in our presentation. However, in the next section we will prove a more general dualization result that requires only minimal technical requirements, and can be applied far beyond the concrete axiomatization of NBG Set Theory.}


\subsection{Dual Null Set}

 For the \emph{Axiom $N^+$(Null Sed)}, let us first note that we can write the classic Axiom $N$ in the following equivalent form:

\[(\exists X)(\forall Y)(\neg(Y\in X)).\]

Effectively, the empty set satisfies clearly the former condition due to the fact that any proper class $Z$ also fulfills $\neg (Z\in \emptyset)$. On the other hand, 

a class $X$ satisfying that any class $Y$ does not belong to it, would fulfill, in particular, the classic condition defining the empty set. Therefore

due to the Class Existence Theorem \cite[Prop. 4.4 Ch 4]{mendelsonlogic}, both should be the same.

So, let us prove that the corresponding dual version of the former version of the Axiom $T$ also holds. 
In fact, 

\[(\exists X)(\forall Y)(\neg(Y\varepsilon X)) \Leftrightarrow\]

\[(\exists X)(\forall Y)(\neg(Y^+\in X^+)) \Leftrightarrow\]

\[(\exists X)(\forall Y^+)(\neg(Y^+\in X^+)) \Leftrightarrow\]

\[(\exists X^+)(\forall Y^+)(\neg(Y^+\in X^+)) \Leftrightarrow\]

\[(\exists X^+)(\forall Z)(\neg(Z\in X^+)) \Leftrightarrow\]

\[(\exists X)(\forall Z)(Z\notin X^+).\]

Now, the last sentence says that there exists a class whose complement is the empty set, which is true because the universal class $V$ of all sets fulfills this property.

So, the \emph{empty sed} is the universal class $V$. So, the \emph{duniversal class} containing all the delements is the empty set.

\subsection{Dual Unordered Pairs}

 We should define a unique value for $\Lbag X,Y\Rbag$, where $X$ and $Y$ are any classes. So, we do this in the natural way: 

\[ Z=\Lbag X,Y\Rbag :\Leftrightarrow Z^+=\{ X^+,Y^+\}.\]

Thus, the \emph{unordered d-pair} is defined as the null sed if one of the classes is a d-proper class, and it is defined by the Axiom $P^+$ if both classes are seds. Besides, by definition we get the equality $(\Lbag X,Y \Rbag)^+=\{ X^+,Y^+\}.$

In addition, we define the \emph{ordered d-pair} of $X$ and $Y$, $\llangle X,Y\rrangle$ as 
\[\Lbag \Lbag X \Rbag, \Lbag X,Y\Rbag\Rbag.\]
It is quite simple to prove that this notion fulfills the corresponding dual property that an ordered pair satisfies, i.e., two ordered d-pairs are equal if and only if the first and the second components coincide. Similarly, one defines ordered d-pairs with $n$ components.
Again, from this definition we can prove that $(\llangle X,Y \rrangle)^+=\langle X^+,Y^+\rangle$.

\subsection{Dual Axiom for the Existence of a Membership Relation}

 The axiom of the existence of the $\varepsilon-$relation states that
\[ (\exists X)(\forall u)(\forall v)(\llangle u,v\rrangle \varepsilon X \leftrightarrow u\varepsilon v).\]

Now, it is equivalent to 

\[ (\exists X^+)(\forall u^+)(\forall v^+)(\langle u^+,v^+\rangle \in X^+ \leftrightarrow u^+\in v^+).\]

So, this sentence shows the existence of a class whose complement is the $\in$-relation class, which is true, since the complement of the $\in$-relation fulfills the statement above. 

\subsection{Dual Existence of Intersections}

 The axiom of the existence of \emph{dintersections} of seds states that 

\[(\forall X)(\forall Y)(\exists Z)(\forall u)(u \varepsilon Z \leftrightarrow u\varepsilon X \wedge u\varepsilon Y).\]

It is equivalent to the following statement:

\[(\forall X^+)(\forall Y^+)(\exists Z^+)(\forall u^+)(u^+ \in Z^+ \leftrightarrow u^+\in X^+ \wedge u^+\in Y^+),\]

where $X^+,Y^+$ and $Z^+$ vary over classes and $u^+$ varies over sets. Now, the last statement is equivalent to the classic Axiom of the existence of the intersection class of two classes. Moreover, if we denote this new class by $Z=X\sqcap Y$, then it holds

\[X\sqcap Y=(X^+ \cap Y^+)^+.\]

Analogously, there is a notion of \emph{dunion} of classes denoted by $X \sqcup X$ satisfying 

\[X \sqcup Y = (X^+ \cup Y^+)^+.\]

\subsection{Dual Notion of Complement}

 The notion of the \emph{domplement} of a sed is given by the statement:

\[ (\forall X)(\exists Z)(\forall u)(u\varepsilon Z \leftrightarrow \neg(u\varepsilon X)),\]
which is equivalent to 

\[ (\forall X^+)(\exists Z^+)(\forall u^+)(u^+\in Z^+ \leftrightarrow u^+\notin X^+)).\]

Where again $X^+$ and $Z^+$ vary over classes and $u^+$ varies over sets. As before, the former sentence is equivalent to the axiom of the existence of the complement class. Besides, if we denote this class by $X^d$, then 

\[X^d=Z=(Z^+)^+=X^+.\]
So, both notions coincides.

Note that due to the definition of equality, all the classes defined before are uniquely determined, which justifies the introduction of the new symbols. 

\subsection{Dual Existence of Domains of Classes}

 The sentence guaranteeing the existence of $d-$domains of classes is the following:

\[(\forall X)(\exists Z)(\forall u)(u\varepsilon Z \leftrightarrow (\exists v)(\llangle u,v \rrangle \varepsilon X).\]

It is equivalent to 

\[(\forall X^+)(\exists Z^+)(\forall u^+)(u^+\in Z^+ \leftrightarrow (\exists v^+)(\langle u^+,v^+ \rangle \in X^+),\]

where $X^+$ and $Z^+$ vary over classes and $u^+$ varies over sets.

This is equivalent to the Domain Existence Axiom. If we denote this new class by $Z=\mathbb{D}^+(X)$, then

\[\mathbb{D}^+(X)=(\mathbb{D}(X^+))^+,\]

where $\mathbb{D}(-)$ denotes the complement of a class.

Now, it is easy to prove the last three dual versions of the Axioms of Class Existence \cite[\S 1 Ch. 4]{mendelsonlogic}, namely

\[ (\forall X)(\exists Z)(\forall u)(\forall v)(\llangle u,v\rrangle \varepsilon Z \leftrightarrow u\varepsilon X),\]

\[ (\forall X)(\exists Z)(\forall u)(\forall v)(\forall w)(\llangle u,v,w\rrangle \varepsilon Z \leftrightarrow \llangle u,w,v\rrangle \varepsilon X),\]

 and

\[ (\forall X)(\exists Z)(\forall u)(\forall v)(\forall w)(\llangle u,v,w\rrangle \varepsilon Z \leftrightarrow \llangle v,w,u\rrangle \varepsilon X).\]

Besides, there is a natural dual notion of difference of sets, defined as \emph{d-difference} of $X$ and $Y$, i.e., 
\[X\leftthreetimes Y:= X \sqcap Y^d.\]

\subsection{Dual Class Existence Theorem}

 The general D-Class Existence Theorem (DCET) is the following: 

Let $\Phi(w_{\alpha},\cdots,w_l,X_1,\cdots,X_m,Y_1,\cdots,Y_n)$ be a dwf formula, where the only relation symbols allowed are $=$ and $\varepsilon$, the bound (seds) variables are exactly $w_{\alpha},\cdots,w_l$, and the free variables occur among $X_1,\cdots,X_m,$ $Y_1,\cdots,Y_n$. Then 

\[\vdash (\exists Z)(\forall x_1)\cdots(\forall x_m)(\llangle x_1,\cdots, x_m\rrangle\varepsilon Z \leftrightarrow \Phi(x_1,\cdots,x_m,Y_1,\cdots,Y_n)).\]

Now, there is a predicative well-formed formula $\Phi^+$, corresponding to $\Phi$, constructed in the following natural way: $\Phi^+$ is obtained from $\Phi$ replacing the relation symbol $\varepsilon$ by $\in$ and replacing the bounded sed variables by bounded set variables with the same names.

So, it can easily be seen that the last sentence is equivalent to the following one:

\[ (\exists Z^+)(\forall x_1^+)\cdots(\forall x_m^+)(\langle x_1^+,\cdots, x_m^+\rangle\in Z^+ \leftrightarrow \Phi^+(x_1^+,\cdots,x_m^+,Y_1^+,\cdots,Y_n^+)),\]

which is a theorem due to the general Class Existence Theorem applied to the predicative wf $\Phi^+(w_{\alpha}^+,\cdots,w_l^+,X_1^+,\cdots,X_m^+,Y_1^+,\cdots,Y_n^+).$

\subsection{Dual Cartesian Product}

 The definition of dual Cartesian product of $X$ and $Y$ is the following:

\[(\forall x)(x\varepsilon X \boxtimes Y \leftrightarrow (\exists u)(\exists v)( x=\llangle u,v\rrangle \wedge u\varepsilon X \wedge v\varepsilon Y)).\]

This class exists in virtue to the DCET and it holds 

\[ X \boxtimes Y = (X^+ \times Y^+)^+.\]

Similarly, one defines Cartesian products for more than two classes. Besides, the dual notion of relation is the concept of \emph{delation}, namely, $X$ is a (binary) delation if $X \sqsubseteq \emptyset^{[2]}:=\emptyset \boxtimes \emptyset$. Based on this concept, we can directly define the dual notions concerning relations,e.g. $X$ is an irreflexive delation of $Y$, $X Irr^+ Y$:

\[Rel^+(X) \wedge (\forall y)(y\varepsilon Y \rightarrow \neg(\llangle y,y \rrangle \varepsilon X)).\]

Similarly, we can define $X Tr^+ Y$ ($X$ is a transitive delation on $Y$), $X Part^+ Y$ ($X$ \emph{partially d-orders} $Y$), $X Con^+ Y$ ($X$ is a connected delation of $Y$) and $X \text{Tot}^+ Y$ ($X$ \emph{totally d-orders} $Y$) and $X We^+ Y$ ($X$ \emph{well-d-orders} $Y$). So, they also fulfill the corresponding dual properties, e.g. $X \text{Tot}^+ Y$ if and only if $X^+ \text{Tot} Y^+$.

In general, the following dual notions exist also due to the DCET:

\subsection{Dual Notion of Power Class}

 By the DCET and the definition of equality, given a class $X$, there is a unique \emph{Dower class} of $X$, denoted by $\mathbb{P}^+(X)$ containing as delements all subseds of $X$:

\[\vdash (\forall X)(\exists_1 Z)(\forall y)(y\varepsilon Z \leftrightarrow y \sqsubseteq X).\]

\subsection{Dual Axiom $U$}

 The Axiom $U^+$ states that 

\[ (\forall x)(\exists y)(\forall u)(u\varepsilon y \leftrightarrow (\exists v)(u\varepsilon v \wedge v\varepsilon x)).\]

As usual, re-writing this statement with the classic notation we can see that it is equivalent to the corresponding Axiom $U$. Besides, we can easily prove that 

for any sed $X$ 

\[ \sqcup X=(\cup (X^+))^+.\]

\subsection{Dual Notion and Axiom of Sum Class}

Similarly, by using the DCET, one proves the existence of a dual notion of Sum class, namely, for any class $X$ there exists a class $Z=\sqcup(X)$ (\emph{the D-Sum Class}) such that 

\[ (\forall y)(y\varepsilon \sqcup X\leftrightarrow (\exists v)(y\varepsilon v \wedge v \varepsilon X)).\]



\subsection{Dual Axiom $W$} Moreover, the dual version for the Axiom $W$ (Power Set), i.e., the Axiom $W^+$ (\emph{Dower Sed}) is the following

\[(\forall x)(\exists y)(\forall u)(u\varepsilon y\leftrightarrow u \sqsubseteq x).\]

This sentence again holds because it is basically equivalent to the standard Axiom W (Power Set).

In particular, it is a formal computation to prove 
\[\vdash \emph{P}^+(\Lbag V, \Lbag V \Rbag \Rbag)=\Lbag V,\Lbag V\Rbag,\Lbag V,\Lbag V\Rbag\Rbag,\Lbag \Lbag V\Rbag\Rbag\Rbag,\]
 
where $V$ denotes the \emph{empty sed}, i.e., the universal class. 

\subsection{Dual Axiom $S$}

 It is a straightforward fact to state and to prove the Axiom $S^+$.

\subsection{Dual Axiom $R$}

 The same holds for the Axiom $R^+$ in terms of a \emph{univocal delation} $Un^+(X)$.

\subsection{Dual Axiom of Infinity}

 The dual version of the Axiom I (Axiom of Infinity) is the Axiom $I^+$ (D-Axiom of infinity):

\[ (\exists x)(V \varepsilon x \wedge (\forall y)(y\varepsilon x \rightarrow y \sqcup \Lbag y\Rbag \varepsilon x)).\]

Informally, it states that there exists a \emph{(dual-)infinite} sed, i.e., a class whose complement is an infinite set.

\subsection{Dual Axiom of Regularity}

 The dual version of the Axiom of Regularity (Axiom D) is the Axiom $D^+$:

 \[\forall X(\exists W(W\varepsilon X) \rightarrow \exists y(y\varepsilon  X \wedge \forall z(z\varepsilon y \rightarrow \neg(z\varepsilon X)))).\]

Again, it can be proved by using the corresponding form of the Axiom of Regularity.

\subsection{Dual Axiom of Choice}

 Finally, there is a dual version of the Axiom of Choice, namely, the \emph{Axiom of D-Choice}, stating that if $x$ is a sed of pairwise \emph{d-disjoint} seds, there exists a (d-choice) sed $c$ containing exactly one denement of each of the seds of $x$. It is a direct consequence of (the corresponding classic version of) the Axiom of Choice.

Equivalently, we can explicitly state the dual form of Zorn's Lemma with the former terminology:

\[ (\forall x)(\forall y)((y\text{Part}x)\wedge(\forall u)(u \sqsubseteq x\wedge y \text{Tot} u \rightarrow(\exists v)(v\varepsilon x\]

\[\wedge (\forall w)(w\varepsilon u\rightarrow w=u \vee \llangle w,v\rrangle \varepsilon y )))\rightarrow (\exists v)(v \varepsilon x \wedge (\forall w)(w \varepsilon x \rightarrow \neg(\llangle v, w \rrangle \varepsilon y)))),\]

which is equivalent to Zorn's Lemma, because all the variables vary over seds and then its dual expressions vary over sets. So, due to the dual properties of each of the delations expressed in the sentence, one basically obtains the corresponding classic form of Zorn's lemma by reading `dually' this specific D-Zorn's Lemma.


Finally, D-NBG Sed theory will be defined as the first-order (sub-) theory (of NBG) having as a logic axioms the standard 5 logic axioms of a first-order theory, as proper axioms we will consider the former dual versions of all the original axioms of NBG and as logical rules (as usual) Modus Ponens and Generalization \cite[Ch. 2 \S 3]{mendelsonlogic}.

\section{A More General Dualization Theorem}

In the last section we developed explicitly all the necessary basic (dual-)notions and facts based essentially on the existence of an unique complementary class, and due to expository reasons as well. Explicitly, The fact that ZFC Set Theory is (one of) the most accepted foundational frameworks for `modern' mathematics compels us to present the minimal explicit results, since our work has direct implications on the identification of the seminal causes of the (in)consistency of such framework.  Now, as the reader may suspect the dual semantic and syntactic core properties lying behind most of the former statements do not depend on the existence of all the axioms of NBG (resp. ZFC).

Effectively, Let $L$ be a first-order language with equality and one binary operation symbol $\swspoon$. Let $T$ be a first-order theory, including an axiom $A_c$ guaranteeing the existence of an unique object, which plays the role of a formal `complement' with respect to $\swspoon$, i.e.,

\[ (\forall a)(a\swspoon X \leftrightarrow a\nswspoon X^d).\]

Let us define a new binary relation symbol $\swfilledspoon$ by the sentence

 \[ (\forall A,B)(A\swfilledspoon B\leftrightarrow A^d\swspoon B^d).\]

Let $L^d$ be the language $L\cup\{\swfilledspoon\}\setminus\{\swspoon\}$. Let $\Phi$ be a $L-$formula. We define the $L^d-$formula $\Phi^{(d)}$ by replacing in $\Phi$ every occurrence of $\swspoon$ by $\swfilledspoon$. 

With the former terminology we can state our general Dualization Theorem.

\begin{theorem}
Let $\Gamma$ be a $L-$theory which includes the axiom $A_c$ and let $\Phi$ be a $L-$sentence. Let $\Gamma^{(d)}$ be the corresponding $L^{(d)}-$theory consisting of the duals of the elements of $\Gamma$. Then, $\Gamma\vdash_L \Phi$ if and only if $\Gamma^{(d)}\vdash_{L^{(d)}}\Phi^{(d)}$. Furthermore, if $M$ is a $L-$model of $\Gamma$, then the natural correspondence induced by the operator $(-)^{(d)}$ induces an isomorphism between $M$ and $M^{(d)}$, where $M^{(d)}$ is exactly $M$ as set and the interpretation of $\swfilledspoon$ is given by means of the original interpretation of $\swspoon$ and the complements. 
\end{theorem} 
\begin{proof}
Let $H_1,\cdots,H_n$ be a $L-$proof of $\Gamma\vdash_L\Phi$ i.e., $H_m=\Phi$ and for any $i=1,\cdots,m-1$, $H_i$ is either an axiom of ZCF, an element of $\Gamma$, or it is a wff that can be deduced by a valid inference rule from the former $H_j-s$
It is straightforward to see that $H_1,\cdots,H_m$ is a $L-$proof of $\Gamma\vdash_L \Phi$, if and only if $H_1^{(d)},\cdots,H_m^{(d)}$ is a $L^{(d)}-$proof of $\Gamma^{(d)}\vdash_{L^{(d)}}\Phi^{(d)}$. On the other hand, one can directly verify that the operator $(-)^{(d)}$ is its own inverse, and therefore it is an isomorphism.

\end{proof}

\section{Dathematics}


Let us call `standard' (or set-theoretic) Mathematics for all formal mathematical theories which are grounded in ZFC set theory, for instance, Real and Complex Analysis, Geometry, Algebra, Number theory, Topology and Category Theory. So, we will name \emph{Dathematics} for the family of all dual versions of the (former) modern theories, where all the subsequent concepts and theorems describing properties among them are expressed and grounded by D-ZFC. 

Here, D-ZFC is, strictly speaking, a first-order logic dual sub-theory of NBG, i.e., in the same way that NBG is a conservative extension of ZFC, so too is D-NBG a conservative extension of (the corresponding theory) D-ZFC.\footnote{The dual NBG first-order logic theory is just NBG itself considered with the former dual axioms (which are in fact, theorems of the theory) and the former conventions about quantification over classes and seds, using explicitly only two binary relation symbols, i.e., $=$ and $\varepsilon$.} In particular, the fundamental objects of Dathematics are (a specific sub-collection of) proper classes, i.e. seds.

Furthermore, the empty-sed is the universal class $V$ and the universal d-proper class $V^+$ is the empty set. So, in Dathematics, the quantitative properties in classical sense are reversed. Besides, it is a natural meta-fact that (classic) mathematics and dathematics are (syntactically) meta-isomorphic, i.e., for any concept, theory and conjecture in (standard) mathematics there exists a symmetric d-concept, d-theory and d-conjecture in dathematics with equivalent formal properties, and vice versa. For instance, we can prove the following syntactic meta-correspondence:


\begin{theorem}
Let $C$ be a conjecture in ZFC (seen as a sub-theory of NBG) given by a wff $\phi$. Then there exists a corresponding dual conjecture $C^+$ in D-ZFC given by the dwf $\phi^+$, such that $C$ is provable in ZFC if and only if $C^+$ is provable in D-ZFC, i.e., $ZFC\vdash \phi$ if and only if $D-ZFC \vdash \phi^+$. Moreover, if $P$ is a proof of $C$ in ZFC, then the natural dual version of $P$ in D-ZFC, namely $P^+$, is a proof of $C^+$ and vice versa. In other words, (standard) Mathematics and Dathematics are meta-isomorphic theories. In particular, they are equiconsistent.
\end{theorem}

\begin{proof}
The argument is similar to the one of the former theorem but with a small additional consideration. Effectively, let $P=\{P_1,\cdots,P_m\}$ be a proof of $\phi$ in ZFC. Then, based on the constructions done in Sect. \ref{sect1}, one can see that each $P_i^+$ is d-wff (i.e. a wff with respect to seds) and $P^+=\{P_1^+,\cdots,P_m^+\}$ is a valid proof of $\phi^+$ in the theory D-ZFC, which has the same inference rules of ZFC. The converse is straightforward. 
So, all (dual) well-formed statements that can be syntactically deduced from one theory, can be (dually) mirrowed into the other one. In conclusion, Mathematics and Dathematics are in this sense syntactically meta-isomorphic.
\end{proof}

In particular, there exists a dual theory of the classical ZFC Set theory which can also be called \emph{ZFC Sed Theory}.

\section{Conclusions}

The fact that set-theoretic Mathematics (based on ZFC (resp. NBG) Set Theory) and Dathematics are meta-isomorphic, and, in particular, one is consistent if and only if the other one is also consistent; together with the fact that the semantics for Dathematics are canonically given by `very big' objects (i.e. proper classes); show that the cause of the Russell's paradox in Naive Set Theory is not only a matter of the `size' of the corresponding foundational objects (e.g. sets), but it also lies within a deeper conceptual level in the formal framework in which sets are defined.

Effectively, we have shown here that there is a formally identical version of standard Mathematics (i.e. Dathematics) structurally based on exactly the same type of objects that turn out to be avoided in NBG because of inconsistency issues, namely, proper classes. In particular, both formal frameworks are `equi-consistent', and both also simultaneously have, from a quantitative perspective, `diametrically opposite' seminal objects, i.e., sets and proper classes.

\section*{Acknowledgements}

The author wishes to thank Pedro Zambrano, Jose Manuel G\'omez, Diana Carolina Montoya and specially to Diego Mejia for the useful suggestions during the elaboration of this manuscript. In addition, he would like to thank Jairo G\'omez for the inspiration by teaching the beauty of formal consistent thinking. Finally, he thanks J. Kieninger for all the support and kindness.

\end{document}